\documentclass[12pt]{amsart}

\usepackage{amssymb}
\usepackage{stmaryrd}
\usepackage{fullpage}
\usepackage{amscd}
\usepackage[all]{xy}
\usepackage{comment}
\usepackage{mathrsfs}
\usepackage{longtable}

\usepackage{tikz}
\usetikzlibrary{positioning, calc}

\usepackage{color}

\renewcommand\mathbb\mathbf
\def\K{\ensuremath{\mathbb{K}}}
\def\kay{\ensuremath{\mathbf{k}}}

\def\cd{\ensuremath{\mathcal{D}}}

\DeclareMathOperator{\Conv}{{Conv}}

\DeclareMathOperator{\codim}{codim}

\newcommand{\Q}{{\mathbb Q}}
\newcommand{\R}{{\mathbb R}}
\newcommand{\Z}{{\mathbb Z}}

\newcommand\td\widetilde

\newcommand{\calA}{{\mathcal A}}

\newcommand{\calO}{{\mathcal O}}

\newcommand{\OO}{{\mathcal O}}

\newcommand{\scrC}{{\mathscr C}}
\newcommand{\scrD}{{\mathscr D}}
\newcommand{\scrE}{{\mathscr E}}
\newcommand{\scrF}{{\mathscr F}}
\newcommand{\scrG}{{\mathscr G}}
\newcommand{\scrH}{{\mathscr H}}

\newcommand{\scrX}{{\mathscr X}}
\newcommand{\scrY}{{\mathscr Y}}

\def\presuper#1#2%
{\mathop{}%
\mathopen{\vphantom{#2}}^{#1}%
\kern-\scriptspace%
#2}

\def\Q{\mathbb{Q}}

\def\R{\mathbb{R}}
\def\P{\mathbb{P}}
\def\A{\mathbb{A}}
\def\Z{\mathbb{Z}}

\def\<{\ensuremath{\langle}}
\def\>{\ensuremath{\rangle}}

\DeclareMathOperator{\Div}{Div}

\DeclareMathOperator{\Spec}{Spec}

\DeclareMathOperator{\val}{val}

\newcommand\djunion\amalg

\numberwithin{equation}{section}
\newtheorem{theorem}[equation]{Theorem}
\newtheorem{lemma}[equation]{Lemma}
\newtheorem{corollary}[equation]{Corollary}

\theoremstyle{definition}
\newtheorem{definition}[equation]{Definition}

\newtheorem{example}[equation]{Example}

\theoremstyle{remark}
\newtheorem{remark}[equation]{Remark}

\usepackage[
	backref,
	pdfauthor={David Brown},
]{hyperref}

\newcommand\bb[1]{\ensuremath{{\mathbb{#1}}}}

\DeclareSymbolFont{symbolsC}{U}{pxsyc}{m}{n}
\DeclareMathSymbol{\coloneq}{\mathrel}{symbolsC}{66}

\renewcommand\bb[1]{\ensuremath{{\mathbf{#1}}}}

\renewcommand\R{\bb R}

\renewcommand\Z{\bb Z}
\renewcommand\Q{\bb Q}

\begin{document}


\title{Newton--Okounkov Bodies over Discrete Valuation Rings and Linear Systems on Graphs}
\author{Eric Katz}
\address{Department of Mathematics, The Ohio State University, 231 W. 18th Avenue, Columbus, OH, 43210, USA}
\email{katz.60@osu.edu}
\author{Stefano Urbinati}
\address{Dipartimento di Matematica, La Nave, Politecnico di Milano, Via Edoardo Bonardi 9, 20133 Milano, Italy}
\email{urbinati.st@gmail.com}


\date{}
\thanks{The second author was supported by the European Commission, Seventh Framework Programme, Grant Agreement n$^{\circ}$ 600376.}

\begin{abstract}
The theory of Newton--Okounkov bodies attaches a convex body to a line bundle on a variety equipped with flag of subvarieties. This convex body encodes the asymptotic properties of sections of powers of the line bundle. In this paper, we study Newton--Okounkov bodies for schemes defined over discrete valuation rings. We give the basic properties and then focus on the case of toric schemes and semistable curves. We provide a description of the Newton--Okounkov bodies for semistable curves in terms of the Baker--Norine theory of linear systems on graphs, finding a connection with tropical geometry. We do this by introducing an intermediate object, the Newton--Okounkov linear system of a divisor on a curve. We prove that it is equal to the set of effective elements of the real Baker--Norine linear system of the specialization of that divisor on the dual graph of the curve. As a bonus, we obtain an asymptotic algebraic geometric description of the Baker--Norine linear system.
\end{abstract}

\maketitle

\section{Introduction}
\label{sec:introduction}

The theory of Newton polytopes \cite[Ch.~6]{GKZ} is a bridge between algebraic geometry and polyhedral geometry. Specifically, if one is given a Laurent polynomial
\[f=\sum_{\omega\in \calA} c_\omega x^\omega\in\kay[x_1^{\pm},\dots,x_n^{\pm}]\]
where $\calA\subset \Z^n$ is a finite set, $\kay$ is a field, and $c_\omega\in\kay^*$, then the Newton polytope of $f$ is the convex hull of $\calA$.  The Newton polytope can be used to understand the intersection theory of the zero locus $Z(f)\subset (\kay^*)^n$ through Bernstein's theorem. There are several ways to generalize Newton polytopes beyond hypersurfaces of the algebraic torus $(\kay^*)^n$. Tropical geometry is one such way, handling higher codimensional subvarieties of tori by studying polyhedral fans instead of polytopes. Another way, the theory of Newton--Okounkov bodies \cite{Okounkov,kk,lm} attaches a convex body to a smooth projective $d$-dimensional algebraic variety $X$ equipped with a divisor $D$ and a flag of subvarieties 
\[Y_{\bullet} : X = Y_{0} \supsetneq Y_{1} \supsetneq ... \supsetneq Y_{d-1} \supsetneq Y_{d} = \{pt\}.\]
 Specifically, the Newton--Okounkov body encodes the generic vanishing of sections of $\OO(mD)$ for $m\in\Z_{\geq 1}$. It is convex and bounded and has many of the desirable properties of Newton polytopes. However, except in low dimensions, it may be non-polyhedral. Moreover, it encodes information about the asymptotic behavior of $\OO(mD)$ making it much more difficult to compute.
 
 Newton polytopes and tropical geometry both have extensions to the ``non-constant coefficient case.'' Specifically, if one has a valued field $\K$ with $\val\colon\K\to\R$ (say, the fraction field of a discrete valuation ring $\OO$), one obtains an unbounded convex body by attaching to
  \[f=\sum_{\omega\in \calA} c_\omega x^\omega\in \K[x_1^{\pm},\dots,x_n^{\pm}],\]
 the {\em upper hull} which is defined to be the convex hull of the set
\[\{(\omega,t)\mid \omega\in\calA,\ t\geq\val(a_\omega)\}\subset \R^n\times\R.\]
 The lower faces of this set induce a subdivision of the Newton polytope $P(f)$ \cite{GKZ, Gubler}. This is the Newton subdivision. The lower faces of the upper hull are the graph of a piecewise linear convex function $\psi\colon P(f)\to \R$. In greater generality, tropical geometry attaches a polyhedral complex to a subvariety $X\subset (\K^*)^n$. 
 
 The purpose of this paper is to consider Newton--Okounkov bodies in the non-constant coefficient case. 
To set notation, let $\OO$ be a discrete valuation ring with fraction field $\K$ and residue field $\kay$.  
Let $\pi$ be a uniformizer of $\OO$.   For convenience, we shall call {\em semistable} any irreducible, regular scheme that is proper, flat, and of finite type over $\OO$ whose generic fiber is smooth and whose closed fiber is a reduced normal crossings divisor.
Let $\scrX$ be a projective regular semistable scheme over $\OO$ with closed fiber $X$.
Let $\scrY_\bullet$ denote a descending flag of subschemes
\[\scrX=Y_0\supsetneq Y_1\supsetneq ...\supsetneq Y_{d+1}\]
 where each $Y_i$ is a codimension $i$ subscheme that is either a semistable scheme over $\OO$ or a smooth subvariety of a component of the closed fiber $\scrX\times_\OO \kay$.   Let $\scrD$ be a  divisor on $\scrX$ flat over $\Spec \OO$.  
 
We will define a Newton--Okounkov body,  $\underline\Delta_{\scrY_\bullet}(\scrD)\subset\R^{d+1}$ by considering sections of $\OO(m\scrD)$ for $m\in\Z_{\geq 1}$ evaluated at a valuation attached to the flag.  In contrast to the classical case but similar to upper hulls, this Newton--Okounkov body will be unbounded, albeit in a single direction.  This unboundedness is a consequence of the fact that if $s\in H^0(\scrX,\OO(\scrD))$, then $\pi^ks\in H^0(\scrX,\OO(\scrD))$ for any $k\geq 0$.  This is formalized by the following theorem:
\begin{theorem}
Let $p_{\pi}:\R^{d+1}\rightarrow \R^d$ be projection along the direction through the valuation vector of $\pi$.  
The image $\Delta=p_{\pi}(\underline{\Delta}_{\scrY_\bullet}(\OO(\scrD))$ is compact.
\end{theorem}

In fact, if $Y_d$ is semistable over $\OO$ and $Y_{d+1}$ is a point of $Y_d\times_\OO \kay$, we are in what we call the {\em tropical case} and more can be said:
\begin{theorem}
In the tropical case, $\underline{\Delta}_{\scrY_\bullet}(\scrD)$ is given by the overgraph of a convex function $\psi\colon p_{\pi}(\underline{\Delta}_{\scrY_\bullet}(\scrD))\rightarrow\R$.
\end{theorem}
Here, the overgraph is the set of points of $p_{\pi}(\underline{\Delta}_{\scrY_\bullet}(\scrD))\times\R$ lying above the graph of $\psi$.

We will give a complete description in the special case of Newton--Okounkov bodies of toric schemes with respect to a toric flag. Here, it is analogous to the field case and involves a particular polyhedron $P_\scrD$  depending on $\scrD$, and $\phi_\R$, a certain linear map depending on $\scrY_{\bullet}$.
\begin{theorem}
Let $\scrD$ be a torus-invariant divisor on a toric scheme $\scrX$ that is flat over $\Spec \OO$ with generic fiber $D$ such that $\OO(D)$ is a big line bundle on $X$.  
Then, the Newton--Okounkov body of $\OO(\scrD)$ is given by
\[\underline{\Delta}_{\scrY_\bullet}(\scrD)=\phi_\R(P_\scrD).\]
\end{theorem}

Newton--Okounkov bodies over $\kay$ in low dimensions turn out to be particularly tractable. In the case of curves, the Newton--Okounkov body is the interval $[0,\deg(D)]$ and so captures	only the degree of the divisor. In the case of surfaces, the Newton--Okounkov body is a polytope encoding the Zariski decomposition of a family of divisors. The case of semistable curves over discrete valuation rings shares features of both of these cases. This is perhaps not surprising because such a curve is of relative dimension $1$ and absolute dimension $2$. We will give a fairly complete description of such Newton--Okounkov bodies in terms of the Baker--Norine theory of linear systems on graphs \cite{BN}. We will phrase our description in language reminiscent of the Zariski decomposition of divisors.

To describe the case of semistable curves, we choose to introduce an intermediate object, the {\em Newton--Okounkov linear system}, $L_\Delta^+(\scrD)$. Like Newton--Okounkov bodies, it measure the asymptotics of vanishing orders of sections of $\OO(m\scrD)$ for $m\in\Z_{\geq 1}$. However, instead of incorporating vanishing orders on a flag, it incorporates vanishing orders 
on components of the special fiber. It is a convex subset of the space of functions $\varphi\colon V(\Sigma)\to \R$ where $\Sigma$ is the dual graph of the closed fiber of $\scrC$. We prove that the Newton--Okounkov linear system is combinatorial by relating it to a combinatorially-defined effective linear system $L^+(\rho(\scrD))$ where $\rho(\scrD)$ is the combinatorial specialization of the horizontal divisor $\scrD$ on the curve $\scrC$:

\begin{theorem} If the generic fiber of $\scrD$ has positive degree, then we have the equality between the Newton--Okounkov linear system and the  effective linear system:
\[L^+_\Delta(\scrD)=L^+(\rho(\scrD)).\]
\end{theorem}

By incorporating the lattice of integer-valued functions $\varphi\colon V(\Sigma)\to\Z$,
this theorem gives an algebraic geometric description of rank in the Baker--Norine theory, a problem studied in work of Caporaso--Len--Melo \cite{CLM}. However, our description works only by studying the asymptotics of sections, and we do not know how to use our results to give algebraic geometric proofs of combinatorial results in the Baker--Norine theory.

After enriching the above theorem by incorporating vanishing orders of sections at a point in the closed fiber of $\scrC$, we are able to give a description of the Newton--Okounkov bodies of curves in Theorem~\ref{t:tropicalcurve} and Theorem~\ref{t:arakeloviancurve}.

We should note that recent, quite different connections between tropical geometry and Newton--Okounkov bodies were recently found by Kaveh--Manon \cite{KM}.

\subsection{Acknowledgments}
We would like to thank David Anderson and Alex K\"{u}ronya for enlightening conversations and helpful comments.

\section{Notation and conventions}

For a scheme $\scrX$ over $\OO$, we will use $\scrX_\K$ to denote its generic fiber. 
 If $\scrC$ is a semistable curve over $\OO$, irreducible divisors are either horizontal or vertical, see, for example, \cite[Sec.~8.3]{Liu}. Here, horizontal divisors are those that that are flat over $\Spec \OO$ while vertical divisors are contained in the closed fiber. Any divisor on $\scrC$ can be decomposed into a sum of horizontal and vertical divisors.

Recall that a convex cone in $\R^d$ is a convex set invariant under rescaling by elements of $\R_{\geq 0}$. For a set $S\subset\R^d$, we write $\text{cone}_{\R^{d+1}}(S)$ for the minimal convex cone containing $S$.

We will write $0$ for the empty divisor. For divisors $D,E$, we will write $D\leq E$ if $E-D$ is effective.
If $D$ is a divisor on a smooth variety $X$, sections of $\OO(D)$ can be interpreted in two ways: as a {\em section of a line bundle} or as a {\em rational function} $s$ whose principal divisor satisfies $(s)+D\geq 0$. Consequently if $D\leq E$, a section of $\OO(D)$ can be interpreted as a section of $\OO(E)$ although its zero locus will differ by $E-D$. We will point out the relevant interpretation by using the words ``section'' or ``rational function.''

\section{Construction of the Newton--Okounkov body}

We recall the construction of Newton--Okounkov bodies in the classical setting. 
Let $X$ be a smooth irreducible projective variety of dimension $d$ over a field $\K$. Given a divisor $D$ on X, we want to construct a convex compact subset of $\R^d$ called the \emph{Newton--Okounkov body of $D$}. A flag of smooth irreducible subvarieties
$$Y_{\bullet} : X = Y_{0} \supsetneq Y_{1} \supsetneq ... \supsetneq Y_{d-1} \supsetneq Y_{d} = \{pt\}$$ 
is called \emph{full and admissible} if $\codim Y_{i} = i$. For every non-zero section $s$ of $\OO(D)$, if $s_0:=s$, for $i=1,\dots, n$ define 
\begin{equation}
	\nu_i(s) := {\rm ord}_{Y_i}(s_{i-1}), \qquad s_i:= \left.\frac{s_{i-1}}{g_i^{\nu_i(s)}}\right|_{Y_i},
\end{equation}
where $g_i$ is the local equation of $Y_i$ in $Y_{i-1}$ near $Y_n$.
Here, $s_0$ is considered as a section of $L_0=\OO(D)$ while $s_i$ is considered as a section on $Y_i$ of $L_i=L_{i-1}|_{Y_i}\otimes \OO_{Y_i}(-\nu_i(\varphi)Y_i)$.  
We obtain the vector $\nu(s) = (\nu_{1}(s), \dots, \nu_{d}(s))\in\Z^n$.
Define the semigroup of valuation vectors as
 \[\Gamma_{Y_\bullet}(D):=  \{ (\nu(s),m) \in \Z^n\times\Z_{\geq 1} \mid s \in  H^{0}(X, \mathcal{O}_{X}(mD))\} \]
and the \emph{Newton--Okounkov body} of $D$ as
\[\Delta_{Y_\bullet}(D) := \overline{\text{cone}_{\R^{d+1}}\left(\Gamma_{Y_\bullet}(D)\right)}\cap  \left(\R^n\times \{1\}\right).\]
We write $\Gamma_{Y_\bullet}(D)_m$ for $\Gamma_{Y_\bullet}(D)\cap (\Z^n\times\{m\})$.
The same construction can be performed for non-complete graded linear series as well, see \cite{kk}. 

\subsection{The case of curves}

%
%
Let us consider a curve $C$ of positive genus $g$ and a divisor $D$ of degree $d>0$. Let $Y_{\bullet}=\{ C \supsetneq p \}$, with $p$ a point. Then the Newton--Okounkov body is the segment $[0, d]$ by  \cite[Ex.~1.2]{lm} as a consequence of the Riemann-Roch theorem. This is a result that we will eventually generalize to the case of semistable curves over a discrete valuation ring.
%

\subsection{The case of surfaces}

In the case of surfaces, the Zariski decomposition of big divisors can be used to show that the Newton--Okounkov body lies between the graphs of two functions on an interval \cite[Sec.~6.2]{lm}. This description provides all the information about possible shapes of Newton--Okounkov bodies of surfaces. 

Any pseudoeffective divisor $D$ (that is, a divisor in the closure of the effective cone in the Neron--Severi group) can be written as $D=P_D +N_D$, where $P_D$ is nef, $P_D\cdot N_D=0$, and $N_D$ is effective with a negative definite intersection matrix.

Let us consider the rank two valuation induced by a general flag $Y_{\bullet}=\{ X \supsetneq C \supsetneq p \}$ such that $p \notin \mbox{supp}(N_D)$.

Let $\alpha(D)= \mbox{ord}_p(N_D)$, $\beta(D)=\alpha(D)+ C\cdot P_D$, and $\mu:= \sup\{x \mid D-xC \mbox{ is big}\}$. Then, we have by the recipe given in \cite{lm}
$$\Delta_{Y_{\bullet}}(D)=\{(x, y) \in \R^2 |\, 0 \leq x \leq \mu \mbox{ and } \alpha(D- xC) \leq y \leq \beta(D-xC)\}.$$

\begin{example} Let $X$ be the blow up of $\mathbb{P}^2$ at two points with exceptional divisors $E_{1}, E_{2}$ and consider the flag  $Y_{\bullet}=\{X \supsetneq l \supsetneq p \}$ given by a general line and a general point on it. Let $H$ denote the pullback of the class of a line in $\P^2$. 

Let $D = 2H - E_{1} - E_{2}$. In this case we have $\mu=1$ and 
the Zariski decomposition of $D-xH$ is the sum $(1-x)(2H-E_1-E_2) + x(H-E_1-E_2)$, obtaining the following body:

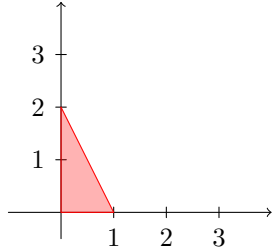
\begin{center}
\begin{tikzpicture}[line cap=round,line join=round,x=.7cm,y=.7cm]
\draw[->,color=black] (-1,0) -- (4,0);
\foreach \x in {1,2,3}
\draw[shift={(\x,0)},color=black] (0pt,2pt) -- (0pt,-2pt) node[below] {\footnotesize $\x$};
\draw[->,color=black] (0,-0.5) -- (0,4);
\foreach \y in {1,2,3}
\draw[shift={(0,\y)},color=black] (2pt,0pt) -- (-2pt,0pt) node[left] {\footnotesize $\y$};
\fill[color=red,fill=red,fill opacity=0.3] (0,0) -- (1,0) -- (0,2) -- cycle;

\draw [color=red] (0,2)-- (1,0);
\draw [color=red] (0,0)-- (1,0);
\draw [color=red] (0,0)-- (0,2);

\end{tikzpicture}
\end{center}
\end{example}

Newton--Okounkov bodies in higher dimensions can be much more complicated and there is no general strategy for writing them down. We will see that for toric schemes with respect to a torus-invariant flag, the Newton--Okounkov body is determined by combinatorics.

\section{Newton--Okounkov bodies over discrete valuation rings}

\subsection{Definition for schemes over discrete valuation rings}
We now describe the case of Newton--Okounkov bodies on schemes over discrete valuation rings.
Let $\OO$ be a discrete valuation ring with fraction field $\K$, residue field $\kay$, and valuation $\val$.  Let $\pi$ be a uniformizer of $\OO$.  Let $\scrX$  be an  $d$-dimensional semistable scheme over $\OO$. We will write $X=\scrX\times_\OO \K$ for the generic fiber of $\scrX$. Let $\scrY_\bullet$ denote a descending flag of proper subschemes
\[\scrX=Y_0\supsetneq Y_1\supsetneq ...\supsetneq Y_{d+1}\]
where each $Y_i$ is a codimension $i$ subscheme that is either a semistable scheme over $\OO$ or a proper smooth subvariety of the closed fiber $\scrX\times_\OO \kay$.   Let $\scrD$ be a divisor on $\scrX$ flat over $\OO$.   

Let $j$ be the index such that $Y_{j-1}$ is semistable over $\OO$ and $Y_{j}$ is a closed subvariety of $\scrX\times_\OO \kay$.  Then $Y_j$ is a component of $Y_{j-1}\times_\OO \kay$.  We will give names to special cases: when $j=1$, we are said to be in the {\em Arakelovian case}; when $j=d+1$, we are in the {\em tropical case}.  Here, the name ``Arakelovian'' is motivated by a construction by Yuan \cite{Y}. The name ``tropical'' is motivated by a fundamental notion in tropical geometry, the Newton subdivision \cite{Gubler} to which our construction specializes in the toric case. So in a certain sense, our work interpolates between tropical geometry and function field Arakelov theory. 

The Newton--Okounkov body $\underline{\Delta}_{\scrY_\bullet}(\scrD)$ is defined as a convex body in $\R^{d+1}$ exactly as above by considering sections of  $\OO(m\scrD)$ over $\scrX$ for $m\in \Z_{\geq 1}$ evaluated at the valuations attached to the flag.

\subsection{Boundedness}

In contrast to the field case, Newton--Okounkov bodies over discrete valuation rings are not bounded. However, the failure of boundedness can be precisely described.

\begin{lemma} Let $\nu(\pi)\in \R^{d+1}$ be the valuation of the uniformizer $\pi\in \OO$, viewed as a rational function on $\scrX$.  The Newton--Okounkov body $\underline{\Delta}_{\scrY_\bullet}(\scrD)$ is closed under positive translations in the $\nu(\pi)$-direction.
\end{lemma}

\begin{proof}
It suffices to show $\Gamma_{\scrY_\bullet}(\scrD)+k(\nu(\pi),0)\subset\Gamma_{\scrY_\bullet}(\scrD)$ for any $k\in\Z_{\geq 0}$.  Any point of $\Gamma_{\scrY_\bullet}(\scrD)$ is of the form $(\nu(s),m)$ for $s\in H^0(\scrX,\calO(m\scrD))$.  Now, $\pi^ks\in H^0(\scrX,\calO(m\scrD))$ and $\nu(\pi^ks)=\nu(s)+k\nu(\pi)$.
\end{proof}

The Newton--Okounkov body is bounded in other directions.  Let $p_{\pi}\colon\R^{d+1}\rightarrow\R^{d+1}/(\R\nu(\pi))$ be the projection along the $\nu(\pi)$-direction
\begin{lemma} The image under projection, $p_{\pi}(\underline{\Delta}_{\scrY_\bullet}(\scrD))$ is bounded.
\end{lemma}

\begin{proof}
Choose an ample divisor  $H$ on $X$ and an ample divisor $h$ on $Y_j$, which is a component of the projective variety  $Y_{j-1}\times_\OO \kay$.      

We will follow \cite[Lemma~1.10]{lm}.
We begin with the following observation: given a divisor $D$ and an irreducible divisor $Y$ on some scheme $Z$, there exists an integer $b$ such that for any section $s_m$ of $\OO(mD)$, the vanishing order of $s_m$ on $Y$ is at most $mb$. Indeed, choose $b$ sufficiently large such that $(D-bY)\cdot H^{d-1}<0$. Multiplying this inequality by $m$, we see that $\OO(mD-mbY)$ cannot have any regular sections.

We claim that there exists positive integers $b_1,\dots,b_{j-1}$ such that for any section $s\in H^0(X,\calO(m\scrD_\K))$, we have
\[\nu_i(s)\leq mb_i,\  \text{for}\ i=1,\dots,j-1.\]
  Given $s\in H^0(\scrX,\calO(m\scrD))$, restriction to the generic fiber gives $s\in H^0(X,\calO(m\scrD_\K))$.  We choose $b_1$ as in the above paragraph. Then, $v_2$ is given by the vanishing order at $Y_2$ of the restriction of $s$ to $Y_1$, considered as a section of $\OO(D_{1,a})=\OO(D)|_{Y_1}\otimes \OO_{Y_1}(-aY_1)$ for some $a$ with $0\leq a\leq b_1$. Choose $b_2$ to be the max of the $b$'s produced in the above paragraph for $D_{1,a}$ for $0\leq a\leq b_1$. We continue by defining $b_i$'s inductively.

Now, there are finitely many line bundles on $Y_{j-1}$ whose sections we will restrict to $Y_j$. We will take the maximum of the $b$'s chosen for all line bundles. Therefore, it suffices to work with one line bundle $\OO(D)$ on $Y_{j-1}$ at a time. Because we are restricting our attention to $Y_{j-1}$, it also suffices to prove the result for $j=1$. Let us consider elements of $H^0(X,\OO(D))$. These are exactly rational sections of $\OO(m\scrD)$ over $\scrX$ that are allowed to have poles along components of the closed fiber. By multiplying such sections by a suitable power of $\pi$, we can ensure that the section is regular.
Therefore, the Newton--Okounkov body associated to $H^0(X,\OO(D))$ with respect to the flag $Y_1\supsetneq ... \supsetneq Y_d$ is exactly the Minkowski sum of the Newton--Okounkov body of $H^0(\scrX,\OO(\scrD))$ and the line $\R \nu(\pi)$. Since we're only interested in the projection of that Newton--Okounkov body along the line through $\nu(\pi)$, it suffices to consider the complementary slice given by elements of $H^0(X,\OO(D))$ that have neither poles nor zeroes generically along $Y_1$.

Now, the sections under consideration restrict to sections on $Y_1$. We may now apply the above argument with $H$ replaced by $h$ to obtain $b_1,\dots,b_d$.
\end{proof}

\begin{remark}
In the tropical case, if $Y_{d+1}$ is a smooth point of the central fiber, $p_{\pi}$ is the projection along the $(d+1)$-st component.
\end{remark}

\subsection{The tropical case}

We now consider the tropical case where the admissible flag $\scrY_\bullet$ is given by semistable schemes $Y_1,\dots,Y_d$ in $\scrX$ and a point $Y_{d+1}$ in the closed fiber $\scrX_\kay$.  Let $\scrD$ be a divisor on $\scrX$ flat over $\Spec \OO$ whose generic fiber is $D\subset X$. 

We will relate the Newton--Okounkov bodies to overgraphs.  Recall that for $\Delta\subset \R^d$, a convex body and a convex function $\psi\colon\Delta\rightarrow\R$, we define its overgraph in $\R^{d+1}$ to be the set
\[\{(x,t)\mid x\in\Delta,\ t\geq\psi(x)\}.\]
If $\psi$ is piecewise linear, its domains of linearity give a subdivision of $\Delta$.

\begin{theorem} We have a surjection of Newton--Okounkov bodies,
\[p_\pi:\underline{\Delta}_{\scrY_\bullet}(\scrD)\to\Delta_{Y_\bullet}(D).\]
Moreover,  $\underline{\Delta}_{\scrY_\bullet}(\scrD)$ is given as the overgraph of a convex function
\[\psi\colon\Delta_{Y_\bullet}(D)\to\R\]
where $\Delta_{Y_\bullet}(D)$ is the Newton--Okounkov body of $D=\scrD\times_{\OO} \K$ on $X=\scrX\times_{\OO} \K$ with respect to 
\[Y_\bullet=\{X\supsetneq Y_1\times_{\OO} \K \supsetneq ... \supsetneq Y_d \times_{\OO} \K\}.\]
\end{theorem}

\begin{proof}
Because any section $s\in H^0(X,\OO(D))$ has some multiple by $\pi$ satisfying $\pi^ks\in H^0(\scrX,\OO(\scrD))$, we have that the projection $p_\pi:\R^{d+1}\to\R^d$ maps $\underline{\Delta}_{\scrY_\bullet}(\scrD)$ surjectively to $\Delta_{Y_\bullet}(D)$.

Now, $\underline{\Delta}_{\scrY_\bullet}(\scrD)$ is closed under positive translation by $e_{d+1}=\nu(\pi)$.  From the convexity of Newton--Okounkov bodies, it follows that $\underline{\Delta}_{\scrY_\bullet}(\scrD)$ is the set of points in $\R^{d+1}$ lying above the graph of a convex function $\psi\colon\Delta_{Y_\bullet}(D)\to \R$.
\end{proof}


We note that Newton--Okounkov bodies of schemes over discrete valuation rings need not be polyhedral. Indeed, one may take a semistable model of a variety with a flag over $\K$ that already has a non-polyhedral Newton--Okounkov body \cite{KLM} and extend the flag to a tropical one.

\begin{remark} Our construction of functions on Newton--Okounkov bodies has some relation to the filtered linear systems that have appeared in the work of Boucksom--Chen \cite{BC}, Witt Nystrom \cite{WN}, and Yuan \cite{Yuan2}. Let $X$ be an algebraic variety or scheme equipped with a flag of subvarieties or subschemes.  Suppose that there is a family of norms $||s||_m$ on sections of the line bundles $L^{\otimes m}=\OO(mD)$ (possibly as the sup-norm coming from a metric on $\OO(D)$). The function on the Newton--Okounkov body is induced by considering the infimum of $\frac{1}{m}\log ||s||_m$ of sections $s$ of $\OO(mD)$  corresponding to a point $\frac{1}{m}\nu(s)$ in the Newton--Okounkov body.  The induced function, called a {\em Chebyshev transform}  is related to metric and adelic volumes in K\"{a}hler and Arakelov geometry. More generally, one may even consider filtrations on sections of $mL$ induced by valuations as in the study vanishing sequences by Boucksom--K\"{u}ronya--Maclean--Szemberg \cite{BKMS}.
\end{remark}

\section{Toric schemes}

In this section, we discuss toric schemes over discrete valuation rings.  See \cite{KKMS} for a classical source or \cite{Gubler} for a rigid analytic perspective.

\subsection{Toric varieties}

We begin by reviewing Newton--Okounkov bodies for smooth projective toric varieties \cite[Section~6.1]{lm}.  Let $N$ be an $d$-dimensional lattice. A toric variety $X(\Delta)$ is specified by a complete rational fan $\Delta$ in $N_\R=N\otimes\R\cong\R^d$. The variety $X(\Delta)$ is smooth if and only if the fan is unimodular, that is, the fan is simplicial and every cone is spanned by integer vectors forming a subset of a basis of $N$. 
Let $T=N\otimes \K^*$ denote the $d$-dimensional algebraic torus acting on $X(\Delta)$. To each $k$-dimensional cone $\sigma$ of $\Delta$, there corresponds an orbit closure $V(\sigma)$ which is a codimension $k$ subvariety.
 Any torus-invariant divisor is given by 
\[D=\sum_\sigma a_\sigma V(\sigma)\]
where the sum is over rays $\sigma$ in $\Delta$ and $a_\sigma\in\Z$. 
Attached to $D$ is a polyhedron $P_D\subset M_\R$ where $M$ is the dual lattice of $N$, defined by
\[P_D=\Conv(\{m\in M \mid \<m,u_\sigma\>\geq -a_\sigma\}).\]
where $u_\sigma\in N$ is the primitive integer vector (with respect to $N$) along $\sigma$.  This polyhedron arises by considering sections of $\OO(D)$: the vector space of sections $H^0(X,\OO(D))$ has a decomposition into $T$-eigenspaces; the lattice points of $P_D$ are exactly the characters of $T$ that arise; for $m\in P_D$, the character $\chi^m$ on $T$ extends to a section of $\OO(D)$ on $X$.  
Indeed, the vanishing order of $\chi^m$ (considered as a section of $\OO(D)$) on the divisor $V(\sigma)$ is $\< m,u_\sigma\>+a_\sigma$ so the inequalities defining $P_D$ are exactly the conditions that $\chi^m$ is regular at the generic point of the torus-invariant divisors. Consequently, $\dim H^0(X,\OO(D))=|P_D\cap M|.$
The line bundle $\OO(D)$ is big if and only if $P_D$ is $d$-dimensional.

Because $X(\Delta)$ is smooth, a $T$-invariant flag $Y_1, Y_2,\dots, Y_n$ can be written as 
\[Y_i=D_1\cap\cdots\cap D_i\]
for a choice of $T$-invariant divisors $D_1,\dots,D_n$ corresponding to rays $\sigma_1,\dots,\sigma_n$. Let $u_1,\dots,u_n$ be the primitive integer vectors along $\sigma_1,\dots,\sigma_n$. We define a linear map $\phi\colon M_\R\to\R^n$ by $\phi(v)=\big(\langle v,u_{
\sigma_i}\rangle+a_{\sigma_i}\big)_{1\leq i\leq n}$. We have the following equality for big line bundles $\OO(D)$:
\[\Delta_{Y_\bullet}(D)=\phi(P_D).\]

\subsection{Toric schemes} 
Complete toric schemes over a discrete valuation ring $\OO$ are described by complete rational fans in $N_{\R}\times \R_{\geq 0}$ where $N\cong \Z^d$ is a lattice.  Given such a fan $\Sigma$, there is a natural morphism of toric varieties $X(\Sigma)_\Z\rightarrow X(\R_{\geq 0})_\Z=\A^1_\Z$ and the toric scheme is given by $\scrX=X(\Sigma)\times_{\A^1} \Spec(\OO)$.  Here, we will map $t$, the coordinate on $\A^1$, to the uniformizer of $\OO$.   We will suppose that $\Sigma$ is a unimodular fan and therefore that the total space $\scrX$ is regular.  If we set $\Delta=\Sigma\cap (N_\R\times\{0\})$, then the generic fiber of $\scrX$ is the toric variety $X=X(\Delta)$. 
The closed fiber of $\scrX$ is a union of toric varieties described combinatorially by the polyhedral complex $\Sigma_1=\Sigma\cap (N_\R\times\{1\})$ in $N_\R\times\{1\}$.  The components of the closed fiber are in bijective correspondence with the vertices of $\Sigma_1$.
 We will suppose that the vertices of $\Sigma_1$ are at points of $N\times\{1\}$ which ensures that $\scrX$ has reduced closed fiber and, therefore, is semistable.  Let $T=N\otimes\K^*$ denote the torus of $X$.  

A $T$-invariant  divisor $D$ on $X$ has many extensions $\scrD$ to $\scrX$.  In particular, we may write $D=\sum_\sigma a_\sigma V(\sigma)$ where $a_\sigma\in\Z$ and  $V(\sigma)$ is the divisor of $X$ corresponding to a ray $\sigma$ of $\Delta$.   Any extension is of the form 
\[\scrD=\sum_\sigma a_\sigma V(\sigma) + \sum_v a_v V(v)\]
where $a_v\in\Z$ and $V(v)$ are the divisors on $\scrX$ corresponding to rays in $N_\R\times\R_{\geq 0}$ through the vertices of $\Sigma_1$.

We construct the Newton--Okounkov body.   Considering the total space $X(\Sigma)$ as an $(d+1)$-dimensional toric variety, we define a polyhedron $P_\scrD\subset M_\R\times \R_{\geq 0}$  by
\[P_\scrD=\Conv\Big(\big\{(m,h)\in M_{\R}\times \R_{\geq 0} \mid \<m,u_\sigma\>\geq -a_\sigma,\ \<m,v\>+h\geq -a_v\big\}\Big).\]
The second set of inequalities come from $v\in\Sigma_1$ corresponding to vertices $(v,1)\in\Sigma_1$.
Note that this projects onto $P_D$ by $(m,u)\mapsto m$.
We may define a piecewise linear convex function on $P_D$,
\[\psi(m)=\max(-a_v-\<m,v\>)\]
where $v$ is taken over vertices of $\Sigma_1$.
Then $P_\scrD$ is the overgraph of $\psi$. 

Now, we will explain how the polyhedron $P_\scrD$ relates to the Newton--Okounkov body of $\scrX$ with respect to a flag $\scrY_\bullet$ of torus-fixed subschemes.  
Following \cite{lm}, we may choose irreducible toric divisors $D_1,\dots,D_{d+1}$ of $X(\Sigma)$ such that $Y_i=D_1\,\cap \cdots\, \cap D_i$.  Suppose that $D_i$ corresponds to a ray in $N_\R\times \R$ whose primitive integer vector is $w_i\in N\times \Z$. Here, $\{w_1,\dots w_{d+1}\}$ is a basis for $N\times\Z$.  
Write 
\[\scrD=\sum a_w V(\sigma_w)\]
where $w$ runs over primitive integer vectors of rays $\sigma_w$ of $\Sigma$.  

We define 
\[\phi_\R\colon M_\R\times \R\to \R^{d+1},\quad \phi\colon(v,h)\mapsto\Big(\big\<(v,h),w_i\big\>+a_{w_i}\Big)_{1\leq i\leq d+1}\]
where the pairing is between $M_\R\times\R$ and $N_\R\times\R$. We have the following analogue of \cite[Prop.~6.1]{lm}:
 
 \begin{theorem} 
Let $\scrD$ be a torus-invariant divisor on $\scrX$ that surjects onto $\Spec \OO$ with generic fiber $D$ such that $\OO(D)$ is a big line bundle on $X$.  
Then,the Newton--Okounkov body of $\OO(\scrD)$ is given by
\[\underline{\Delta}_{\scrY_\bullet}(\OO(\scrD))=\phi_\R(P_\scrD).\]
 \end{theorem}

\begin{proof}
By replacing $\scrD$ with a positive integer multiple, we may suppose that the vertices of $P_\scrD$ are points of $N\times\Z$.

Write the restriction of $s\in H^0((\scrX,\OO(\scrD)))$ to $X$ as
\[s=\sum_m c_m \chi^m\]
for $c_m\in \K$ where the above is a finite sum over characters. The vanishing order of $s$ (considered as a rational function) on the divisor $D_w$ corresponding to a ray $\sigma_w$ of $\Sigma$ is 
\[b_w=\min\Big(\Big\<\big(m,\val(c_m)\big),w\Big\>\Big)\]
where $\val(0)=\infty$
and the pairing is the one between $M_\R\times \R$ and $N_\R\times \R$. Observe that in the above, if $\sigma_w$ is a ray of $\Delta$, then $\<(m,\val(c_m)),w\>=\<m,u_\sigma\>_N$ where the second pairing is the pairing between $M_\R$ and $N_\R$. If $w=(v,1)$ corresponds to a vertex of $\Sigma_1$, then  $\Big\<\big(m,\val(c_m)\big),w\Big\>=\<m,v\>_N+\val(c_m)$. On $\scrX$, we have that following formula for the principal divisor:
\[(s)=\sum b_w D_w.\]
Note that the vanishing order of $s$, considered as a section of $\OO(\scrD)$, on $D_w$ is $b_w+a_w$.

Consequently, a sum of characters like the above corresponds to an element of $H^0((\scrX,\OO(\scrD))$ if and only $b_w\geq -a_w$ for all $w$. In fact, $\pi^h\chi^m\in H^0((\scrX,\OO(\scrD))$ if and only if $(m,h)\in P_\scrD$.

Now, the valuation of such a section of $\OO(\scrD)$ is
\[\nu(s)=(b_{w_1}+a_{w_1},\dots,b_{w_{d+1}}+a_{w_{d+1}}).\]
It follows that $\nu(s)\in\phi_\R(P_\scrD)$. By considering sections of the form $\pi^h\chi^m$, we see that $\underline{\Delta}_{\scrY_\bullet}(\scrD)$ contains $\phi_\R(P_\scrD)$.
 \end{proof}

\section{Newton--Okounkov bodies of curves}

We will relate the Newton--Okounkov bodies of curves over $\OO$ to the Baker-Norine theory of linear systems on graphs.  

\subsection{Review of linear systems on graphs}

We review some results on specialization of linear systems from curves to graphs due to Baker \cite{B}.    Let $\scrC$ be a semistable curve over $\Spec \OO$.  The semistability condition ensures that the closed fiber $\scrC_0$ is reduced with only ordinary double points as singularities.  A node in the closed fiber of $\scrC$ is formally locally described in $\scrC$ by $\OO[x,y]/(xy-\pi)$.

\begin{definition} The {\em dual graph} $\Sigma$ of a  semistable curve $\scrC$ is a graph $\Sigma$ whose vertices $V(\Sigma)$ correspond to  components of the normalization $\pi:\widetilde{\scrC}_0\rightarrow\scrC_0$ and whose edges $E(\Sigma)$ correspond to nodes of $\scrC_0$. For each vertex $v\in V(\Sigma)$, we write $C_v$ for the corresponding component of $\widetilde{\scrC}_0$. 
\end{definition}

We will denote the edges of $E(\Sigma)$ by $e=vw$ even though $\Sigma$ may not be a simple graph. Thus, when we sum over edges adjacent to $v$, we may need to sum over certain vertices more than once and sum over $v$ itself.

A {\em divisor} on $\Sigma$ is an element of the real vector space with basis $V(\Sigma)$.  We write a divisor as $D=\sum_{v\in V(\Sigma)} a_v(v)$ with $a_v\in\R$.  We may write $D(v)=a_v$. The vector space of all divisors is denoted by $\Div(\Sigma)$.  We say a divisor $D$ is effective and write $D\geq 0$ if $a_v\geq 0$ for all $v\in V(\Sigma)$.  We write $D\geq D'$ if $D-D'\geq 0$.  The degree of a divisor is given by
\[\deg(D)=\sum_v a_v.\]
We will study functions $\varphi\colon V(\Sigma)\to \R$. The Laplacian of $\varphi$, $\Delta(\varphi)$ is the divisor on $\Sigma$ given by
\[\Delta(\varphi)=\sum_{v\in V(\Sigma)} \sum_{e\in E(\Sigma) \mid e=vw} (\varphi(v)-\varphi(w))(v).\]
Note that $\Delta(\varphi)$ is of degree $0$.

The specialization map $\rho\colon\Div(\scrC)\rightarrow\Div(\Sigma)$ is defined by,
for  $\cd\in\Div(\scrC)$,
\[\rho(\cd)=\sum_{v\in\Gamma}\deg(\pi^*\OO(\cd)|_{C_v})(v).\]
The specialization of a vertical divisor $\sum_v \varphi(v)C_v$ satisfies
\[\rho\left(\sum_v \varphi(v) C_v\right)=-\Delta(\varphi).\]
For a divisor $H$ on $C_\K$, we will write $\rho(H)$ to mean the specialization of its closure in $\scrC$.  Observe that for $H$, horizontal and effective, we have $\rho(H)\geq 0$.

\begin{definition} Let $\Lambda$ be a divisor on $\Sigma$.  We define the {\em linear system} $L(\Lambda)$ to be the set of functions $\varphi\colon V(\Sigma)\to \R$ on $\Sigma$ with $\Delta(\varphi)+\Lambda\geq 0$.  The {\em effective linear system} $L^+(\Lambda)$ is the subset of $L(\Lambda)$ consisting of everywhere non-negative functions $\varphi$.
\end{definition}

Let $\scrD$ be a divisor on $\scrC$.  Then we will interpret a global section of $\calO(\scrD)$ as a rational function $s$ on $\scrC$ such that $(s)+\scrD\geq 0$.  If we write $(s)=\scrH+V$ where $\scrH$ is a horizontal divisor over $\calO$ and $V$ is a vertical divisor contained in the closed fiber, we may decompose $V$ as
$V=\sum_v \varphi_s(v) C_v$
where we call  $\varphi_s\colon V(\Sigma)\rightarrow \Z$ the {\em vanishing function} of $s$.  For $s$, a rational function on $C$, we will abuse notation and take the vanishing function of $s$ to be vanishing function of the extension of $s$ to $\scrC$.

The following lemma is standard and we include the proof only for completeness.

\begin{lemma} \label{l:spec} Let $D$ be a divisor on $C_\K$ whose closure $\scrD$ has specialization $\Lambda=\rho(\scrD)$.  For  a rational function $s$ on $C$ corresponding to a section of $\OO(D)$ with vanishing function $\varphi$, we have
\[\Delta(\varphi)+\Lambda\geq 0\]
or, in other words, $\varphi\in L(\Lambda)$.
\end{lemma}

\begin{proof}
Because $s$ is principal, we have $(s)\cdot C_v=0$ for all components of the closed fiber.  If we write $(s)=\scrH+\sum_v \varphi(v) C_v$, we have 
\[0=\rho\left((s)\right)=\rho\big(\scrH+\sum_v \varphi(v) C_v\big)=\rho(\scrH)-\Delta(\varphi).\]
Since $\scrH+D\geq 0$ is horizontal, we have
\[0\leq \rho(\scrH)+\rho(D)=\Delta(\varphi)+\rho(D).\]
\end{proof}
 
The linear system $L(\Lambda)$ has a tropical semigroup structure as noted in \cite{HMY}:

\begin{lemma} \label{l:semimodule} For $\varphi_1,\varphi_2\in L(\Lambda)$, let $\varphi\colon V(\Sigma)\to \R$ be the pointwise minimum of $\varphi_1,\varphi_2\colon V(\Sigma)\to \R$.  Then $\varphi\in L(\Lambda)$.
\end{lemma}

\begin{proof}
Let $v\in V(\Sigma)$.   Without loss of generality, suppose that $\varphi(v)=\varphi_1(v)$.  Then
\begin{eqnarray*}
\Delta(\varphi)(v)+\Lambda(v)&=&\sum_{e=vw} (\varphi(v)-\varphi(w))+\Lambda(v)\\
&\geq &\sum_{e=vw} (\varphi_1(v)-\varphi_1(w))+\Lambda(v)\\
&\geq & 0.
\end{eqnarray*}
\end{proof}

\subsection{Geometric and tropical linear systems}

\begin{definition}
Now, let $\scrD$ be a horizontal divisor on $\scrC$.  Let $m\in \Z_{\geq 1}$.  There is a natural map 
\begin{eqnarray*}
\varrho_m\colon H^0(\scrC,\OO(m\scrD))&\to &L^+(\rho(\scrD))\\
s&\mapsto& \frac{1}{m}\varphi_s
\end{eqnarray*}
where $s\in H^0(\scrC,m\scrD)$ is interpreted as a rational function $s$ with $(s)+m\scrD\geq 0$
and $\varphi_s$ is the vanishing function of $s$.
\end{definition} 
 
We define the {\em Newton--Okounkov linear system} $L^+_\Delta(\scrD)$ to be the subset of $L^+(\rho(\scrD))$ given by the closure of the union of the convex hulls of the images of $\varrho_m$ for $m$ ranging over $\Z_{\geq 1}$. We may extend this definition to horizontal $\Q$-divisors by defining $L^+_\Delta(\scrD)$ to be $\frac{1}{m}L^+_\Delta(m\scrD)$ where $m$ is chosen arbitrarily divisible.
 
\begin{theorem} \label{p:nols} If the generic fiber of $\scrD$ has positive degree, then we have the equality between the Newton--Okounkov linear system and the effective linear system:
\[L^+_\Delta(\scrD)=L^+(\rho(\scrD)).\]
\end{theorem}

Before proving the proposition, we need a preparatory lemma adapted from \cite{KRZB}.

\begin{definition} Let $f\colon V(\Sigma)\rightarrow \R$ be a function.  Set 
\[M(f)=\max_{S\subseteq V(\Sigma)} \Big\{ \big| \sum_{v\in S} f(v) \big| \Big\}.\]
\end{definition}

\begin{lemma} Let $\varphi\colon V(\Sigma)\rightarrow \R$ be a function.  Then,
\[\max \varphi-\min \varphi\leq M(\Delta(\varphi))\operatorname{diam}(\Sigma).\]
\end{lemma}

\begin{proof}
It suffices to show that for any edge $e=vw$ in $\Sigma$, $|\varphi(w)-\varphi(v)|\leq M(\Delta(\varphi))$.  Indeed, let $v_0,v_1$ be the vertices where the minimum and maximum of $\varphi$ are achieved, respectively.  By picking a path from $v_0$ to $v_1$ of length at most $\operatorname{diam}(\Sigma)$ and comparing the values of $\varphi$ along that path, we achieve the desired conclusion.

Let $e=v'w'$ with $t=\varphi(v')<\varphi(w')$.  
Set $\Sigma_{\leq t}$ be the subgraph of $\Sigma$ induced by $\varphi^{-1}([0,t])$.     Let $O(\Sigma_{\leq t})$ be the set of outgoing edges, that is, the edges $e=vw\in E(\Sigma)$ with $v\in\Sigma_{\leq t}$ and $w\not\in\Sigma_{\leq t}$. Observe that for such edges $e=vw$, we have $\varphi(w)-\varphi(v)>0$.
Now, 
\begin{eqnarray*}
M(\Delta(\varphi))&\geq& \Big|\sum_{v\in V(\Sigma_{\leq t})} \Delta(\varphi)(v)\Big|\\
&=&\Big|\sum_{v\in V(\Sigma_{\leq t})} \Big(\sum_{e=vw} \big(\varphi(v)-\varphi(w)\big)\Big)\Big|\\
&=&\sum_{e=vw\in O(\Sigma_{\leq t})} (\varphi(w)-\varphi(v))
\end{eqnarray*}
where the last equality holds because the contribution from edges contained in $\Sigma_{\leq t}$ cancel in pairs.
From this we conclude that for any outgoing edge $e=vw$, $\varphi(w)-\varphi(v)\leq M(\Delta(\varphi))$.  
\end{proof}

We have the following corollary:
\begin{corollary} \label{c:bound}
Let $\varphi,\vartheta:V(\Sigma)\rightarrow\R_{\geq 0}$ be functions such that
\[\Delta(\varphi)-\Delta(\vartheta)=F-G\]
where $F$ and $G$ are effective divisors of degree at most $d$.  Suppose that there are (not necessarily distinct) vertices $v,w$ such that $\varphi(v)=\vartheta(w)=0$
Then 
\[\max(|\varphi-\vartheta|)\leq d\operatorname{diam}(\Sigma).\]
\end{corollary}

\begin{proof}
By hypothesis, $\min(\varphi-\vartheta)\leq \varphi(v)-\vartheta(v)\leq 0$.
We note that $M(\Delta(\varphi-\vartheta))\leq d$.  Consequently, 
\[\max(\varphi-\vartheta)\leq \max(\varphi-\vartheta)-\min(\varphi-\vartheta)\leq d\operatorname{diam}(\Sigma).\]
Interchanging the roles of $\varphi$ and $\vartheta$, we get the conclusion.
\end{proof}

We now prove Proposition~\ref{p:nols}.
\begin{proof}
Set $\Lambda=\rho(\scrD)$.  From Lemma~\ref{l:spec}, it follows that $L^+_\Delta(\scrD)\subseteq L^+(\Lambda)$.  

Therefore, we must show that any $\vartheta\in L^+(\Lambda)$ can be approximated by some element in the image of $\rho_m$ for some positive integer $m$. First, we may suppose that $\vartheta$ takes rational values.  Pick a sufficiently divisible $m$ such that $m\vartheta$ takes integer values.  Because $\varrho_m(\pi^k s)=\varrho_m(s)+\frac{k}{m}$, we can replace $\vartheta$ by $\vartheta-\min_v \vartheta(v)$ and suppose that $\vartheta\geq 0$ with $\vartheta(v)=0$ for some vertex $v$.

We have that  $\Delta(m\vartheta)+m\Lambda\geq 0$. 
 Pick an effective divisor $E$ on $\Sigma$ such that $E(v)\in\Z$ for all $v\in V(\Sigma)$ and $\Delta(m\vartheta)+m\Lambda-E$ is effective of degree exactly $g$, the genus of $\scrC$.  Choose a horizontal divisor $\scrE$ such that $\rho(\scrE)=E$.   Because $\deg(m\scrD_{\K}-\scrE_{\K})=g$,  the line bundle $\OO(m\scrD_{\K}-\scrE_{\K})$ on $\scrC_{\K}$ has a regular section by the Riemann-Roch theorem.  Therefore, there is a rational function $s$ on $\scrC_{\K}$ with 
\[(s)+m\scrD_{\K}-\scrE_{\K}\geq 0.\]

From $s$ we will find a section whose image under $\rho_m$ approximates $\vartheta$. By multiplying $s$ by some power of $\pi$, we may ensure that $s$ (considered as a rational function on $\scrC$) is regular on the generic points of the components of the closed fiber and does not vanish identically on all of them.
 Consequently, $s$'s vanishing function $\varphi_s$ is non-negative and takes the value $0$ at some vertex $w$.  
Now, 
\begin{eqnarray*}
\Delta(\varphi_s)-\Delta(m\vartheta)&=&(\Delta(\varphi_s)+m\Lambda-E)-(\Delta(m\vartheta)+m\Lambda-E)\\
&=& (\Delta(\varphi_s)+\rho(m\scrD-\scrE))-(\Delta(m\vartheta)+m\Lambda-E)
\end{eqnarray*}
is the difference of two effective degree $g$ divisors by Lemma~\ref{l:spec}.  Consequently, by Corollary~\ref{c:bound}, we have 
\[
\left|\frac{\varphi_s}{m}-\vartheta\right|\leq \frac{g}{m}\operatorname{diam}(\Sigma).\]
Because $\varrho_m(s)=\frac{\varphi_s}{m}$, the conclusion follows by choosing large $m$.
\end{proof} 
  
To handle the case where the divisor $\scrD$ is of degree $0$, we may employee the following result.

\begin{corollary}
Let $\scrD$ be a horizontal divisor on $\scrC$ of non-negative degree. Let $\scrE$ be an effective, non-empty horizontal divisor. Then, we have the equality
\[L^+(\rho(\scrD))=\bigcap_{\varepsilon>0} L^+_\Delta(\scrD+\varepsilon\scrE)\]
where the intersection is taken over rational $\varepsilon>0$.
\end{corollary}

\begin{proof}
It suffices to prove that
\[L^+(\rho(\scrD))=\bigcap_{\varepsilon>0} L^+(\rho(\scrD+\varepsilon \scrE)).\]
This follows immediately from definitions.
\end{proof}

This comparison between the purely combinatorial Baker--Norine linear system and the algebraically-defined Newton--Okounkov linear system was surprising to these authors. However, it does not capture the combinatorial richness of the Baker--Norine theory as it involves real-valued, rather than integer-valued functions $\varphi$ on graphs. The integer-valued functions can be incorporated into our work by hand. 
Within the vector space of functions $\varphi\colon V(\Sigma)\to\R$, there is a lattice $\varphi\colon V(\Sigma)\to\Z$. 
In \cite{BN}, a divisor $\Lambda$ on a graph $\Sigma$ is said to have non-negative rank if there exists $\varphi\colon V(\Sigma)\to\Z$ such that $\Delta(\varphi)+\Lambda\geq 0$. From this concept, a Riemann--Roch theory for divisors on graphs is developed. Because we may add a constant to $\varphi$ without affecting $\Delta(\varphi)+\Lambda\geq 0$, we may suppose that $\varphi$ is non-negative in the above definition. From this, we can give a asymptotic formulation of non-negative rank.

\begin{corollary}
Let $\scrD$ be a horizontal divisor on $\scrC$ of non-negative degree. Let $\scrE$ be an effective, non-empty horizontal divisor. Then, the specialization $\rho(\scrD)$ has non-negative rank if and only if there exists $\varphi\colon V(\Sigma)\to\Z_{\geq 0}$ such that for all $\varepsilon>0$,
$\varphi\in  L^+_\Delta(\scrD+\varepsilon\scrE)$.
\end{corollary}

From this, one may reformulate the Baker--Norine theory in terms of lattice points in Newton--Okounkov linear systems. It is unknown at this point whether this view leads to any new proofs of known results in the Baker--Norine theory.
 
\subsection{Horizontal-Vertical decomposition}

Now, we will define a decomposition of divisors on $\Sigma$ analogous to the Zariski decomposition to use in our description of Newton--Okounkov bodies of curves.  Recall that the Zariski decomposition of a big $\Q$-divisor $D$ on a smooth projective surface $X$ is a particular decomposition of the linear equivalence class of $D$, $D=P+N$ where $P$ is nef and $N$ is effective.  It has the property that for $m$ such that $mD$ and $mN$ are integral divisors, multiplication by $mN$ gives an isomorphism
\[H^0(X,mP)\to H^0(X,mD).\]

\begin{definition} Let $\Lambda$ be a divisor on $\Sigma$ such that $L^+(\Lambda)$ is non-empty.  The {\em minimal element} of $L^+(\Lambda)$, $\varpi\colon V(\Sigma)\to \R$ is defined by
\[\varpi(v)=\min(\varphi(v)\mid \varphi\in L^+(\Lambda)).\]
\end{definition}

We have the following straightforward lemma following from Lemma~\ref{l:semimodule}.

\begin{lemma} Let $\varpi$ be the minimal element of $L^+(\Lambda)$.  Addition of $\varpi$ gives an isomorphism 
\[L^+(\Lambda-\Delta(\varpi))\to L^+(\Lambda).\]
\end{lemma}

We can interpret $\Lambda=(\Lambda-\Delta(\varpi))+\Delta(\varpi)$ as a sort of Zariski decomposition.

Moreover, if $L\subseteq L^+(\Lambda)$ is a sub-semigroup, we may define $\varpi_L$ to be the pointwise minimum of $\varphi\in L$.

\subsection{Enriched Newton--Okounkov linear systems}

We will connect the Newton--Okounkov linear systems to the Newton--Okounkov bodies of curves over discrete valuation rings. Such bodies must take into account the vanishing of sections along a flag 
\[\scrY_\bullet=\{\scrC=Y_0\supsetneq Y_1\supsetneq Y_2\}\] 
where $Y_2=\{p\}$ is a  smooth point of the closed fiber.  Consequently, we will enrich the above theory by considering such vanishing.
We will also need to consider elements of  $H^0(\scrC,m\scrD)$  whose horizontal components do not have any components in common with a fixed horizontal divisor in order to gain control over the vanishing at $Y_1$ in the tropical case.

Let $\scrD,\scrF$ be horizontal divisors on $\scrC$.  Let $m\in \Z_{\geq 1}$.

\begin{definition} The {\em $\scrF$-controlled linear system} $H^0(\scrC,m\scrD)_{(\scrF,\varepsilon)}$ is the set of all $s\in H^0(\scrC,m\scrD)$ which, when considered as rational functions,  have the property that their principal divisor $(s)$ contains no component of $\scrF$ with multiplicity greater than $m\varepsilon$. 
\end{definition}

Let $p$ be a smooth point on a component $C_v$ of the closed fiber  of $\scrC$. For a divisor $\scrG$ on $\scrC$, write $v_p(\scrG)$ to be the multiplicity of $p$ in $\scrH\cap C_v$ where $\scrH$ is the horizontal part of $\scrG$.
We consider the natural map
\begin{eqnarray*}
\varrho_{m,p}\colon H^0(\scrC,m\scrD)_{(\scrF,\varepsilon)}&\to &L^+(\rho(\scrD))\times \R\\
s&\mapsto& \left(\frac{1}{m}\varphi_s,\frac{1}{m}v_p\big((s)+m\scrD\big)\right)
\end{eqnarray*}
where $s\in H^0(\scrC,m\scrD)_{\scrF,\varepsilon}$. Observe that the second component of $\varrho_{m,p}(s)$ is the vanishing at $p$ of the horizontal component of the zero locus of $s$, considered as a section of $\OO(\scrD)$.

\begin{definition}
The {\em $\scrF$-controlled $p$-enriched Newton--Okounkov linear system} $L^+_{\Delta,p}(\scrD)_{(\scrF,\varepsilon)}$ is the subset of $L^+(\Lambda)\times \R$ given by the closure of the union of the convex hulls of the images of $H^0(\scrC,m\scrD)_{(\scrF,\varepsilon)}$ under $\varrho_{m,p}$ for $m\in\Z_{\geq 1}$.  When the subscript $(\scrF,\varepsilon)$ is suppressed, this means that we consider the empty divisor.
\end{definition}

For a divisor $\Lambda$ on $\Sigma$, let the {\em p-enriched effective linear system} $L^+_p(\Lambda)$ be the subset of $L^+(\Lambda)\times \R$ given by
\[L^+_p(\Lambda)=\{(\varphi,u)\mid \varphi\in L^+(\Lambda),\ 0\leq u\leq \Delta(\varphi)(v)+\Lambda(v)\}.\]
Observe that if $\Lambda=\rho(\scrD)$ for a horizontal divisor $\scrD$ on $\scrC$, the quantity $\Delta(\varphi)(v)+\Lambda(v)$ is exactly the degree of the divisor $\sum_v \varphi(v)C_v+\scrD$ restricted to $C_v$. So, the second component measures the possible multiplicities of $p$ in the horizontal part of $(s)$ varying from $0$ to the maximal possible degree.

We have the following extension of Proposition~\ref{p:nols}.

\begin{theorem} \label{t:exnols} Let $\scrC$ be a semistable curve over $\Spec \OO$  with horizontal divisors $\scrD$ and $\scrF$ such that the generic fiber of $\scrD$ has positive degree. Let $p$ be a smooth point on a component of $C_v$ of the closed fiber of $\scrC$.  We have the equality between the $\scrF$-controlled $p$-enriched Newton--Okounkov linear system and the $p$-enriched effective linear system:
\[L^+_{\Delta,p}(\scrD)_{(\scrF,\varepsilon)}=L_p^+(\rho(\scrD)).\]
\end{theorem}

This is proved by the same method as Proposition~\ref{p:nols}. We wish to find a section $s\in H^0(\scrC,m\scrD)$ such that $\varrho_{m,p}(s)$ is close to some $(\vartheta,u)$ We modify the proof to choose  $\scrE$ to intersect $C_v$ with a multiplicity $m'$ close to $mu$ at $p$ and for $\scrE$ not to contain any component of $\scrD$ or $\scrF$.  The section $s$ produced by the Riemann-Roch theorem has
$(s)+m\scrD_{\K}-\scrE_{\K}$ equal to a degree $g$ divisor. Then, the horizontal component of $(s)+m\scrD_{\K}$ intersects $C_v$ with a multiplicity between $m'$ and $m'+g$. Moreover, the horizontal components of $(s)+m\scrD$ supported on components of $\scrF$ are of degree at most $g$. By picking a sufficiently large $m$, we obtain a close approximation.

\subsection{Newton--Okounkov bodies of curves}

In this section, we give a combinatorial description of the Newton--Okounkov bodies of curves.

We first consider the tropical case of a flag $\{\scrX\supsetneq Y_1 \supsetneq Y_2\}$ where $Y_1$ is a horizontal divisor and $Y_2=\{p\}$ is a smooth point of the closed fiber.  

\begin{theorem} \label{t:tropicalcurve} Suppose that $Y_1$ is a horizontal divisor intersecting the closed fiber in smooth points and that $Y_2=\{p\}$ is  a point on the component $C_v$.   Let $\scrD$ be a horizontal divisor whose generic fiber has positive degree. Moreover, suppose that $p$ is not contained in $\scrD$. For $t\in \R$, let 
\[L_t=L^+(\rho(\scrD-t Y_1)).\]
The Newton--Okounkov body is the overgraph in $\R^2$ of 
\begin{align*}
a\colon [0,\deg(\scrD)/\deg(Y_1)]&\rightarrow \R\\
t&\mapsto\varpi_{L_t}(v).
\end{align*}
\end{theorem}

\begin{proof}
We first show that the Newton--Okounkov body lies above the graph of $a$.
We observe that for $s\in H^0(\scrC,m\scrD)$, $\nu_2(s)$ is always greater than or equal to the multiplicity of $C_v$ in the principal divisor $(s)$.  Consequently, if $t\in\R$ is such that  $mt$ is an integer and $s$ vanishes to order at  least $mt$ on $Y_1$, then $\varrho_m(s)\in L_t$.  Therefore, 
\[
\frac{1}{m}\nu_2(s)\geq \frac{1}{m}\varphi_s(v)\geq \min(\varphi(v)\mid \varphi\in L_t)=\varpi_{L_t}(v).
\]

For a fixed rational $t$ with $0\leq t<\deg(\scrD)$, let us find an $m\in\Z_{\geq 1}$ and a section $s\in H^0(\scrC,\OO(m\scrD))$ such that $\frac{1}{m}\nu(s)$ is close to $(t,\varpi_{L_t}(v))$. Set $\scrF=Y_1$.  For a small $\varepsilon>0$, by Theorem~\ref{t:exnols}, we may find $m$ large enough such that there exists $s\in H^0(\scrC,m(\scrD-t\scrF))_{(\scrF,\varepsilon)}$ (considered as a rational function) such that 
\[\frac{1}{m}v_p\big((s)+m(\scrD-t\scrF)\big)<\varepsilon\]
and $\frac{1}{m}\varphi_s$ is within $\varepsilon$ of $\varpi_{L_t}$.   It follows that $\frac{1}{m}\nu_1(s)$ is close to $t$. From the fact that
\[\nu_2(s)=\frac{1}{m}\big(\varphi_s(v)+v_p\left((s)+m\scrD-\nu_1(s)\scrF)\right)\big),\]
it follows that all but a small part of $v_2(s)$ comes from the vertical component of $(s)$ along $C_v$. Thus $v_2(s)$ can be made arbitrarily close to $\varpi_{L_t}(v)$.
\end{proof}

Now, we consider the Arakelovian case.

\begin{theorem} \label{t:arakeloviancurve} Let $\scrD$ be a horizontal divisor.  Suppose that $Y_1$ is a component $C_v$ of the closed fiber and $Y_2=\{p\}$ is a smooth point on $C_v$ not contained in $\scrD$. Let $L_t\subset L^+(\rho(D))$ be the tropical sub-semigroup of elements $\varphi$ with $\varphi(v)=t$.
Then the Newton--Okounkov body of $\calO(D)$ is the set of points between the graphs of $a(t)=0$ and 
\[b(t)=\rho(\scrD)(v)+\max(\Delta(\varphi)(v)\mid \varphi\in L_t)\]
for $t\geq 0$.
\end{theorem}

\begin{proof}
Observe that for $s\in H^0(\scrC,m\scrD)$, $\nu_1(s)=\varphi_s(v)$ and 
\[\nu_2(s)=v_p\big((s)+m\scrD\big).\]
The Newton--Okounkov body is the image of $L^+_{\Delta,p}(\scrD)=L_p^+(\rho(\scrD))$ under the map $(\varphi,u)\mapsto (\varphi(v),u)$.  The conclusion follows from Theorem~\ref{t:exnols}.
\end{proof}

\begin{example}
We conclude by giving an example of the Newton--Okounkov body for curves in the tropical and Arakelovian cases for the same linear system.

Let us consider the example in \cite[Section 4.4]{B} of a smooth plane quartic curve of $X$ with genus $g(X)=3$. This is a plane quartic degenerating into a conic $C$ and two lines $\ell_1,\ell_2$. To make the model semistable, one must blow up the intersection point of $\ell_1$ and $\ell_2$, introducing a new component $E$ of the degeneration.
 For this curve, the special fiber and the dual graph are given in the figure:  the vertex $P$ corresponds to the conic; $Q_1,Q_2$ corresponds to the lines; and $P'$ corresponds to the curve $E$.

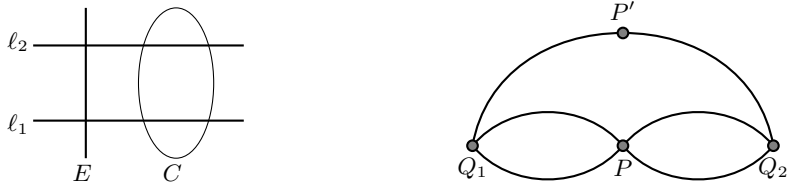
\begin{figure}[h!]
\begin{center}
   \begin{tikzpicture}
    \draw[thick] (-1.4,.5) -- (1.4,.5);
    \draw[thick] (-1.4,-.5) -- (1.4,-.5);
    \draw[thick] (-.7,1) -- (-.7,-1);
    \node at (-1.6,-.55)%
{\fontsize{9}{4}\selectfont $\ell_1$};
    \node at (-1.6,.55)%
{\fontsize{9}{4}\selectfont $\ell_2$};
    \node at (.45,-1.2)%
{\fontsize{9}{4}\selectfont  $C$};
    \node at (-.75,-1.2)%
{\fontsize{9}{4}\selectfont $E$};


\draw (.5,0) ellipse (.5cm and 1cm);

   \end{tikzpicture}
\quad \hspace{2cm}
\begin{tikzpicture}[thick]
\draw[thick] (0,0) to [out=50,in=130] (2,0);
\draw[thick] (0,0) to [out=-50,in=-130] (2,0);
\draw[thick] (2,0) to [out=50,in=130] (4,0);
\draw[thick] (2,0) to [out=-50,in=-130] (4,0);
\draw[thick] (0,0) to [out=80,in=180] (2,1.5);
\draw[thick] (4,0) to [out=100,in=0] (2,1.5);
\draw [fill=gray] (0,0) circle (2pt); 
\draw [fill=gray] (2,0) circle (2pt); 
\draw [fill=gray] (4,0) circle (2pt); 
\draw [fill=gray] (2,1.5) circle (2pt); 

\node at (0,-.3) {\fontsize{9}{4}\selectfont $Q_1$};
\node at (4,-.3) {\fontsize{9}{4}\selectfont $Q_2$};
\node at (2,-.3) {\fontsize{9}{4}\selectfont $P$};
\node at (2,1.8) {\fontsize{9}{4}\selectfont $P'$};
\end{tikzpicture}
\end{center}
\end{figure}

We will compute the Newton--Okounkov body for a general hyperplane section $\scrD$, whose specialization is given by $\rho({\scrD})=\Lambda=2(P)+(Q_1)+(Q_2)$. Note that for any $\varphi \in L(\Lambda)$ we have 
\begin{eqnarray*}
\Delta(\varphi)&=&(4\varphi(P) -2\varphi(Q_1)-2\varphi(Q_2))(P)\\
&+&(3\varphi(Q_1) -\varphi(P')-2\varphi(P))(Q_1)\\
&+&(3\varphi(Q_2) -\varphi(P')-2\varphi(P))(Q_2)\\
&+&(2\varphi(P') -\varphi(Q_1)-\varphi(Q_2))(P').
\end{eqnarray*}
\begin{itemize}
\item {\bf Tropical case:} we will pick as a flag $\{\scrC=Y_0\supsetneq Y_1 \supsetneq Y_2\}$, with $Y_1$ a degree one horizontal divisor, intersecting the generic fiber $\scrX_\K$ in a general point and intersecting the closed fiber in a generic point $Y_2$ of the conic $C$.  It is straightforward to compute $\varpi_{L_t}$ for $t\in [0,4]$ as follows:
\begin{itemize} 
\item for $t\in [0,2]$, $\varpi_{L_t}=0$,
\item for $t\in [2,4]$, $\varpi_{L_t}(P)=(t-2)/4$,\ $\varpi_{L_t}(Q_1)=\varpi_{L_t}(Q_2)=\varpi_{L_t}(P')=0.$
\end{itemize}
This gives the following Newton--Okounkov body:

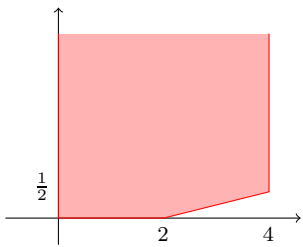
\begin{center}
\begin{tikzpicture}[line cap=round,line join=round,x=.7cm,y=.7cm]
\draw[->,color=black] (-1,0) -- (4.6,0);
\draw[->,color=black] (0,-0.5) -- (0,4);
\fill[color=red,fill=red,fill opacity=0.3] (0,0) -- (2,0) -- (4,.5) -- (4,3.5) -- (0,3.5) -- cycle;


\node (a) at (2,-.3) {$_2$};
\node (b) at (4,-.3) {$_4$};
\node (b) at (-.3,.6) {$_{\frac{1}{2}}$};

\draw [color=red] (0,3.5)-- (0,0);
\draw [color=red] (0,0)-- (2,0);
\draw [color=red] (2,0)-- (4,.5);
\draw [color=red] (4,.5)-- (4,3.5);

\end{tikzpicture}
\end{center}

\item{\bf Arakelovian case:} we will pick as a flag $\{\scrC=Y_0\supsetneq Y_1 \supsetneq Y_2\}$ where $Y_1$ is the conic $C$ and $Y_2$ is a general point of $C$. The function $b(t)$ is achieved by the following choices for $\varphi$:
\begin{itemize}
\item for $t\in [0,1/2]$, $\varphi(P)=t,\ \varphi(Q_1)=\varphi(Q_2)=\varphi(P')=0$,
\item for $t\in [1/2,\infty)$, $\varphi(P)=t,\ \varphi(Q_1)=\varphi(Q_2)=\varphi(P')=t-1/2.$
\end{itemize}

Therefore, the Newton--Okounkov body is as follows:
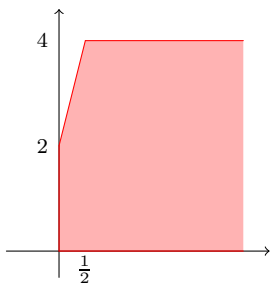
\begin{center}
\begin{tikzpicture}[line cap=round,line join=round,x=.7cm,y=.7cm]
\draw[->,color=black] (-1,0) -- (4,0);
\draw[->,color=black] (0,-0.5) -- (0,4.6);
\fill[color=red,fill=red,fill opacity=0.3] (0,0) -- (3.5,0) -- (3.5,4) -- (.5,4) -- (0,2) -- cycle;

\draw [color=red] (.5,4)-- (3.5,4);
\draw [color=red] (0,0)-- (0,2);
\draw [color=red] (0,2)-- (.5,4);
\draw [color=red] (0,0)-- (3.5,0);

\node (a) at (.5,-.35) {$_{\frac{1}{2}}$};
\node (b) at (-.3,4) {$_4$};
\node (b) at (-.3,2) {$_2$};

\end{tikzpicture}
\end{center}

\end{itemize}

\end{example}

The two examples are reflections of each other because in the tropical case, we chose $Y_1$ with $\rho(Y_1)=P$, the vertex corresponding to the conic while in the Arakelovian case, we chose $Y_1=C$, the conic. In general, the Newton--Okounkov bodies in the two cases will not have such an obvious relation to each other.


\begin{thebibliography}{99}



\bibitem{B}
M. Baker,
\emph{Specialization of linear systems from curves to graphs},
\newblock{Algebra Number Theory} {\bf 2} (2008), no. 6, 613--653.

\bibitem{BN}
M. Baker and S. Norine, 
\emph{Riemann--Roch and Abel--Jacobi theory on a finite graph},
Adv. Math. \textbf{215} (2007), 766--788.

\bibitem{BC}
S. Boucksom and H. Chen, 
\emph{Okounkov bodies of filtered linear series,}
Compos. Math. {\bf 147} (2011), no. 4, 1205--1229. 

\bibitem{BKMS}
S. Boucksom, A. K\"{u}ronya, Alex, C. Maclean, and T. Szemberg, 
\emph{Vanishing sequences and Okounkov bodies,} 
Math. Ann. {\bf 361} (2015), no. 3--4, 811--834. 


\bibitem{CLM}
L. Caporaso, Y. Len, M. Melo, Margarida
\emph{Algebraic and combinatorial rank of divisors on finite graphs.}, 
J. Math. Pures Appl. {\bf 104} (2015), no. 2, 227--257. 


\bibitem{GKZ}
I. M. Gelfand, M. M. Kapranov, and A. V. Zelevinsky,
\emph{Discriminants, resultants, and multidimensional determinants,}  Birkh\"{a}user, Boston, 1994.

\bibitem{Gubler}
W. Gubler, \emph{A guide to tropicalizations,}  Algebraic and combinatorial aspects of tropical geometry, Contemp. Math. \textbf{589} (2013), 125--190.

\bibitem{HMY}
C. Haase, G. Musiker, and J. Yu,
\emph{Linear systems on tropical curves,}
Math. Z. \textbf{270} (2012), no. 3--4, 1111--1140.

\bibitem{KRZB}
E. Katz, J. Rabinoff, and D. Zureick-Brown,
\emph{Uniform bounds for the number of rational points on curves of small Mordell--Weil rank},
Duke Mathematical Journal, to appear, \textsf{arXiv:1504.00694}.


\bibitem{kk}
K. Kaveh and A. G.Khovanskii, \emph{Newton--Okounkov bodies, semigroups of integral points, graded algebras and intersection theory},  Ann. of Math. (2) {\bf 176} (2012), no. 2, 925-978.

\bibitem{KM}
K. Kaveh and C. Manon, \emph{Khovanskii bases, Newton--andOkounkov polytopes and tropical geometry of projective varieties}, preprint, \textsf{arXiv:1610.00298}.

\bibitem{KKMS}
G. Kempf, F. Knudsen and D. Mumford, {\em Saint-Donat, B. Toroidal embeddings. I}, Lecture Notes in Mathematics, {\bf 339}. Springer-Verlag, Berlin-New York, 1973.

\bibitem{KLM} A. K\"uronya, V. Lozovanu and C. Maclean, \emph{Convex bodies appearing as Okounkov bodies of divisors}, Adv. Math. 229 (2012), no. 5, 2622--2639. 


\bibitem{lm}
R. Lazarsfeld and M. Musta{\c{t}}{\u{a}}, \emph{Convex bodies associated to linear series}, Ann. Sci. \'Ec. Norm. Sup\'er. (4) {\bf 42} (2009), no. 5, 783--835.

\bibitem{Liu}
Q. Liu,
\emph{Algebraic geometry and arithmetic curves.} Oxford University Press, Oxford, 2002.

%

\bibitem{Okounkov}
A. Okounkov, \emph{Brunn-Minkowski inequality for multiplicities}, Invent. Math. {\bf 125} (1996), no. 3, 405--411. 

\bibitem{WN}
D. Witt Nystr\"{o}m,
{\em Transforming metrics on a line bundle to the Okounkov body,}
Ann. Sci. \'{E}c. Norm. Sup\'{e}r. (4) {\bf 47} (2014), no. 6, 1111--1161. 

\bibitem{Y}
X. Yuan, {\em On volumes of arithmetic line bundles,} Compos. Math. {\bf 145} (2009), no. 6, 1447--1464. 

\bibitem{Yuan2}
X. Yuan, {\em On volumes of arithmetic line bundles II,} preprint, \textsf{arXiv:0909.3680}.

\end{thebibliography}
\end{document}